\documentclass[preprint,9pt]{elsarticle}
\usepackage{amsfonts}
\usepackage{amssymb}
\usepackage{amsmath}
\usepackage{amsthm}
\usepackage{enumerate}

\newtheorem{thrm}{Theorem}[section]
\newtheorem{lem}[thrm]{Lemma}
\newtheorem{prop}[thrm]{Proposition}
\newtheorem{exam}[thrm]{Example}
\newtheorem{cor}[thrm]{Corollary}

\theoremstyle{definition}
\newtheorem{definition}[thrm]{Definition}
\newtheorem{remark}[thrm]{Remark}

\journal{...}
\input{tcilatex}

\begin{document}

\begin{frontmatter}


\cortext[cor1]{Corresponding author (+903562521616-3087)}

\title{Normal Soft int-Groups}


\author[]{Kenan KAYGISIZ \corref{cor1}}
\ead{kenan.kaygisiz@gop.edu.tr}
\address{Department of Mathematics, Faculty of Arts and Sciences,
Gaziosmanpa\c{s}a University, 60250 Tokat, Turkey}

\begin{abstract}
In this paper, we define normal soft int-groups and derive their
some basic properties. We also investigate some relations on $\alpha
$-inclusion, soft product and normal soft int-groups. Then we define
normalizer, quotient group and give some theorems concerning these
concepts.
\end{abstract}

\end{frontmatter}



\section{Introduction}

Zadeh \cite{zad-65} introduced the notion of a fuzzy set in 1965 to deal
with problems that contains uncertainties. In 1971, Rosenfeld \cite{ros-71}
defined the fuzzy subgroup of a group. Rosenfeld's groups made important
contributions to the development of fuzzy abstract algebra. Since then,
various works have studied analogues of results derived from classical group
theory, such as \cite{abo-93, ajm-92, akg-88, ant-79, asa-91, bhu-93,
das-81, dix-90, kim-94, kum-92, liu-82, mor-94, muj-86, muk-84}. All above
papers and others are combined by Mordeson et al. \cite{mordeson-05} in the
book titled Fuzzy Group Theory.

In 1999, Molodtsov \cite{mol-99} introduced another theory, called soft
sets, in order to deal with uncertainty. Then, Maji et al. \cite{maj-03}
defined the operations of soft sets, \c{C}a\u{g}man and Engino\u{g}lu \cite%
{cag-10} modified these definition and operations of soft sets and Ali et
al. \cite{ali-09} gave some new algebraic operations for soft sets. Sezgin
and Atag\"{u}n \cite{sez-11} analyzed the operations of soft sets. Using
these definitions, researches have been very active on the soft set theory
and many important results have been obtained in theoretical and practical
aspects.

The works of the algebraic structure of soft sets was first started by Akta%
\c{s} and \c{C}a\u{g}man \cite{akt-07}. They presented the notion of the
soft group and derived its some basic properties. Jun \cite{jun-08-1} and
Jun and Park \cite{jun-08-2} introduced soft BCK/BCI-algebras and its
application in ideal theory. Feng et al. \cite{fen-08} worked on soft
semirings, soft subsemirings, soft ideals, idealistic soft semirings and
soft semiring homomorphisms. Jun et al. \cite{jun-10} introduced notions of
soft ordered semigroup, soft ideal, and idealistic soft ordered semigroup
with their related properties. Acar et al. \cite{aca-10} gave the notion of
soft rings and investigated their properties.

\c{C}a\u{g}man et al. \cite{cag-11} gave a new kind of definition of soft
group (soft int-group) in a soft set depending on inclusion relation and
intersection of sets. This definition is completely different then the
definition of soft group in \cite{akt-07}. Kayg\i s\i z \cite{kay-11}
presented some supplementary properties of soft sets and soft int-groups and
gave some relations on $\alpha $-inclusion, soft product and soft int-groups.

In this paper, we present normal soft int-groups and investigate their
related properties. We then obtained some relations on $\alpha $-inclusion,
soft product and normal soft int-groups. In addition, we defined normalizer,
quotient group and give some theorems concerning these concepts.

\section{Preliminaries}

\subsection{\protect\bigskip Soft sets}

\bigskip In this section, we present basic definitions of soft set theory
according to \cite{cag-10}. For more detail see the papers \cite%
{maj-03,mol-99}.

Throughout this work, $U$ refers to an initial universe, $E$ is a set of
parameters and $P(U)$ is the power set of $U$. $\subset $ and $\supset $
stands for proper subset and superset, respectively.

\c{C}a\u{g}man and Engino\u{g}lu \cite{cag-10} modified the definition of
soft set defined by Molodtsov \cite{mol-99} as follows;

\begin{definition}
\label{A 10}\cite{mol-99} For any subset $A$ of $E$, a soft set $f_{A}$ over
$U$ is a set, defined by a function $f_{A}$, representing a mapping%
\begin{equation*}
f_{A}:E\longrightarrow P(U)\text{ such that }f_{A}(x)=\phi \text{ if }%
x\notin A\text{.}
\end{equation*}%
A soft set over $U$ can also be represented by the set of ordered pairs
\begin{equation*}
f_{A}=\left\{ (x,f_{A}(x)):x\in E,\text{ }f_{A}(x)\in P(U)\right\} \text{.}
\end{equation*}
\end{definition}

Note that the set of all soft sets over $U$ will be denoted by $S(U) $. From
here on, "soft set" will be used without over $U$. 

\begin{definition}
\label{A 20}\cite{cag-10} Let $f_{A}\in S(U)$. If $f_{A}(x)=\phi $ for all $%
x\in E$, then $f_{A}$ is called an empty soft set and denoted by $\Phi _{A}$.

If $f_{A}(x)=U$ for all $x\in A$, then $f_{A}$ is called $A$-universal soft
set and denoted by $f_{\widetilde{A}}$.

If $f_{A}(x)=U$, for all $x\in E$, then $f_{A}$ is called a universal soft
set and denoted by $f_{\widetilde{E}}$.\bigskip
\end{definition}

\begin{definition}
\cite{maj-03} If $f_{A}\in S(U)$, then the image (value class) of $f_{A}$ is
defined by Im$\left( f_{A}\right) =\left\{ f_{A}(x):x\in A\right\} $ and if $%
A=E$, then Im$\left( f_{E}\right) $ is called image of $E$ under $f_{A}$.
\end{definition}

\begin{definition}
\cite{kay-11} Let $f_{A}:E\longrightarrow P(U)$ be a soft set and $%
K\subseteq E.$ Then, the image of set $K$ under $f_{A}$ is defined by $%
f_{A}(K)=\underset{x\in K}{\bigcup }f_{A}(x).$
\end{definition}

Note that $f_{A}(K)=\phi $ if $K=\phi .$



\begin{definition}
\label{A 30}\cite{cag-10} Let $f_{A},f_{B}\in S(U)$. Then, $f_{A}$ is a soft
subset of $f_{B}$, denoted by $f_{A}\widetilde{\subseteq }f_{B}$, if $%
f_{A}(x)\subseteq f_{B}(x)$ for all $x\in E$.

$f_{A}$ is called a soft proper subset of $f_{B}$, denoted by $f_{A}%
\widetilde{\subset }f_{B}$, if $f_{A}(x)\subseteq f_{B}(x)$ for all $x\in E$
and $f_{A}(x)\neq f_{B}(x)$ for at least one $x\in E$.

$f_{A}$ and $f_{B}$ are called soft equal, denoted by $f_{A}=f_{B}$, if and
only if $f_{A}(x)=f_{B}(x)$ for all $x\in E$.
\end{definition}


\begin{definition}
\label{A 40}\cite{cag-10} Let $f_{A},f_{B}\in S(U)$. Then, union $f_{A}%
\widetilde{\cup }f_{B}$ and intersection $f_{A}\widetilde{\cap }f_{B}$ of $%
f_{A}$ and $f_{B}$ are defined by
\begin{equation*}
\left( f_{A}\widetilde{\cup }f_{B}\right) (x)=f_{A}(x)\cup f_{B}(x),\text{ }%
\left( f_{A}\widetilde{\cap }f_{B}\right) (x)=f_{A}(x)\cap f_{B}(x)
\end{equation*}%
for all $x\in E$, respectively.
\end{definition}


\subsection{\protect\bigskip Definitions and basic properties of soft
int-groups}

In this section, we introduce the concepts of soft int-groups, $\alpha $%
-inclusion, soft product and their basic properties according to papers by
\c{C}a\u{g}man et al. \cite{cag-11} and Kayg\i s\i z \cite{kay-11}.

Note that definitions and propositions are changed according to our
notations.


\begin{definition}
\label{B 10}\cite{cag-11} Let $G$ be a group and $f_{G}\in S(U)$. Then, $%
f_{G}$ is called a soft intersection groupoid over $U$ if $%
f_{G}(xy)\supseteq f_{G}(x)\cap f_{G}(y)$ for all $x,y\in G.$

$f_{G}$ is called a soft intersection group (int-group) over $U$ if the soft
intersection groupoid satisfies $f_{G}(x^{-1})=f_{G}(x)$ for all $x\in G.$
\end{definition}

\bigskip Note that the condition $f_{G}(x^{-1})=f_{G}(x)$ is equivalent to $%
f_{G}(x^{-1})\supseteq f_{G}(x)$ for all $x\in G.$\

Throughout this paper, $G$ denotes an arbitrary group with identity element $%
e$ and the set of all soft int-groups with parameter set $G$ over $U $ will
be denoted by $S_{G}(U)$, unless otherwise stated. For short, instead of $%
``f_{G}$ is a soft int-group with the parameter set $G$ over $U$"\ we say $%
``f_{G}$ is a soft int-group". 

\begin{thrm}
\label{B 20}\cite{cag-11} Let $f_{G}\in S_{G}(U)$. Then $f_{G}(e)\supseteq
f_{G}(x)$ for all $x\in G$.
\end{thrm}


\begin{definition}
\label{B 70}\cite{cag-11} Let $G$ be a group. Then, the soft set $f_{G}$ is
called an Abelian soft set if $f_{G}(xy)=f_{G}(yx)$ for all $x$,$y\in G$.
\end{definition}


\begin{definition}
\label{B 90}\cite{cag-11} Let $f_{G}\in S(U)$. Then, $e$-set of $f_{G}$,
denoted by $e_{f_{G}}$, is defined as
\begin{equation*}
e_{f_{G}}=\left\{ x\in G:f_{G}(x)=f_{G}(e)\right\} .
\end{equation*}
\end{definition}


\begin{thrm}
\label{B 100}\cite{cag-11} If $f_{G}\in S_{G}(U)$, then $e_{f_{G}}\leq G$.
\end{thrm}


We modified the extension principal defined in \cite{cag-11} as follows. 
%

\begin{definition}
\bigskip \label{B 111}\cite{cag-11}Let $\varphi $ be a function from $A$
into $B$ and $f_{A},f_{B}\in S(U)$. Then, soft image of $f_{A}$ under $%
\varphi $ and soft pre-image (or soft inverse image) of $f_{B}$ under $%
\varphi $ are the soft sets $\varphi (f_{A})$ and $\varphi ^{-1}(f_{B})$
such that%
\begin{equation*}
\varphi (f_{A})\left( y\right) =\left\{
\begin{array}{cc}
\bigcup \left\{ f_{A}(x):x\in A,\text{ }\varphi (x)=y\right\} & \text{if }%
\varphi ^{-1}(y)\neq \phi \\
\phi & \text{otherwise}%
\end{array}%
\right.
\end{equation*}%
for all $y\in B$ and $\varphi ^{-1}(f_{B})(x)=f_{B}\left( \varphi (x)\right)
$ for all $x\in A$, respectively.
\end{definition}

Clearly, $\varphi (f_{A})=f_{B}$ such that $f_{B}(y)=f_{A}\left( \varphi
^{-1}(y)\right) $ for all $y\in B$ and $\varphi ^{-1}(f_{B})=f_{A}$ such
that $f_{A}\left( x\right) =f_{B}\left( \varphi (x)\right) $ for all $x\in A$%
.


\begin{thrm}
\label{B 210}\cite{kay-11} If $f_{G}\in S_{G}(U)$ and $H\leq G$, then the
restriction $f_{G}|_{H}\in S_{H}(U)$.
\end{thrm}


\begin{thrm}
\label{B 220}\cite{kay-11} Let $A_{i}\leq G$ for all $i\in I$ and $%
\{f_{A_{i}}:i\in I\}$ be a family of soft int-groups. Then, $\widetilde{%
\bigcap\limits_{_{i\in I}}}f_{A_{i}}\in S_{G}(U)$.
\end{thrm}

\begin{definition}
\label{B 290}\cite{cag-11} Let $f_{A}\in S(U)$ and $\alpha \in P(U)$. Then, $%
\alpha $-inclusion of the soft set $f_{A}$, denoted by $f_{A}^{\alpha },$ is
defined as
\begin{equation*}
f_{A}^{\alpha }=\left\{ x\in A:f_{A}(x)\supseteq \alpha \right\} .
\end{equation*}
\end{definition}

\bigskip If $\alpha =\phi $ then the set\ $f_{A}^{\ast }=\left\{ x\in
A:f_{A}\left( x\right) \neq \phi \right\} $ is called support of $f_{A}$.

In particular, the set $f_{A}^{\alpha ^{\prime }}=\left\{ x\in
A:f_{A}(x)\supset \alpha \right\} $ is called strong $\alpha $-inclusion.


\begin{cor}
\label{B 300}\cite{kay-11} For any soft set $f_{A}$, if $\alpha \subseteq
\beta $ and $\alpha $,$\beta \in P(U)$, then $f_{A}^{\beta }\subseteq
f_{A}^{\alpha }.$
\end{cor}

%

\begin{definition}
\label{B 350}\cite{kay-11} Let $f_{A}\in S(U)$ and $\alpha \in P(U)$. Then,
the soft set $f_{A\alpha }$,\ defined by, $f_{A\alpha }(x)=\alpha $, for all
$x\in A$,\ is called $A-\alpha $ soft set. If $A$ is a singleton, say $\{w\}$%
, then $f_{w\alpha }$ is called a soft singleton (or soft point). If $\alpha
=U$, then $f_{AU}=f_{\widetilde{A}}$ is the characteristic function of $A$.
\end{definition}


%

\begin{thrm}
\label{B 367}\cite{kay-11}Let $f_{G}\in S(U)$ and $\alpha \in P(U).\ $Then,$%
\ f_{G}\in S_{G}(U)$ if and only if $f_{G}^{\alpha }$ is a subgroup of $G$,
whenever it is nonempty.
\end{thrm}

\begin{definition}
\cite{kay-11} Let $f_{G}\in S_{G}(U)$. Then, the subgroups $f_{G}^{\alpha }$
are called soft level subgroups of $G$ for any $\alpha \in P(U)$.
\end{definition}

%

\begin{definition}
\label{B 370}\cite{kay-11} Let $G$ be a group and $A,B\subseteq G$. Then,
soft product of soft sets $f_{A}$ and $f_{B}$ is defined as
\begin{equation*}
(f_{A}\ast f_{B})(x)=\bigcup \{f_{A}(u)\cap f_{B}(v):uv=x,\text{ }u,v\in G\}
\end{equation*}%
for all $x\in G$ and inverse of $f_{A}$ is defined as%
\begin{equation*}
f_{A}^{-1}(x)=f_{A}(x^{-1})
\end{equation*}%
for all $x\in G.$
\end{definition}

%
%

\begin{thrm}
\label{B 380}\cite{kay-11} Let $G$ be a group and $f_{x\alpha }$,$f_{y\beta
}\in S_{G}(U).$\ Then, for any $x$,$y\in G$ and $\phi \subset \alpha $,$%
\beta \subseteq U$, $f_{x\alpha }\ast f_{y\beta }=f_{\left( xy\right) \left(
\alpha \cap \beta \right) }.$
\end{thrm}

%
%
%

\begin{thrm}
\label{B 400}\cite{kay-11} Let $A$,$B$,$C\subseteq G$ and $f_{A}$,$f_{B}$,$%
f_{C}\in S(U).$ Then,
\begin{equation*}
\left( f_{A}\ast f_{B}\right) \ast f_{C}=f_{A}\ast \left( f_{B}\ast
f_{C}\right)
\end{equation*}
\end{thrm}

%
%
%

\begin{cor}
\label{B 420}\cite{kay-11} Let $A\subseteq G$ and $f_{A}$,$f_{u\alpha }$ be
soft sets where $\alpha =f_{A}(A).$ Then, for any $x$,$u\in G$,
\begin{equation*}
(f_{u\alpha }\ast f_{A})(x)=f_{A}(u^{-1}x)\text{ and }(f_{A}\ast f_{u\alpha
})(x)=f_{A}(xu^{-1})
\end{equation*}
\end{cor}

%
%

\begin{corollary}
\label{B 430}Let $f_{A}\in S_{G}(U)$. Then,
\begin{equation*}
f_{A}\widetilde{\subseteq }f_{A}^{-1}\Leftrightarrow f_{A}^{-1}\widetilde{%
\subseteq }f_{A}\Leftrightarrow f_{A}=f_{A}^{-1}.
\end{equation*}%
$.$
\end{corollary}



\begin{thrm}
\label{B 480}\cite{kay-11} Let $f_{G}\in S(U).$ Then, $f_{G}$ is a soft
int-group if and only if $f_{G}$ satisfies the following conditions:

\begin{enumerate}
\item $\left( f_{G}\ast f_{G}\right) \widetilde{\subseteq }f_{G}$,

\item $f_{G}^{-1}=f_{G}$ (or $f_{G}\widetilde{\subseteq }f_{G}^{-1}$ or $%
f_{G}^{-1}\widetilde{\subseteq }f_{G}$).
\end{enumerate}
\end{thrm}


\begin{thrm}
\label{B 490}\cite{kay-11} Let $f_{A}$,$f_{B}\in S_{G}(U)$. Then, $\left(
f_{A}\ast f_{B}\right) $ is a soft int-group\ if and only if $f_{A}\ast
f_{B}=f_{B}\ast f_{A}$.
\end{thrm}


\section{Normal soft int-groups and cosets}

\bigskip Normal subgroup is very important concept in classical group
theory. In this section, we introduce the notion of normal soft int-groups,
and obtain the analogues to the classical group theory and fuzzy group
theory.


\begin{definition}
\label{C 10}Let $f_{A}\in S_{G}(U)$. Then, $f_{A}$ is called soft normal in $%
G$ (or normal soft int-group), if $f_{A}$ is an Abelian soft set.\bigskip
\end{definition}

Let $NS_{G}(U)$ denotes the set of all normal soft int-groups in $G$.


\begin{cor}
\label{C 15} The following conditions are equivalent:

\begin{enumerate}
\item $f_{A}$ is a normal soft int-group,

\item $f_{A}$ is soft Abelian,

\item $f_{A}$ is constant in the conjugate class of $A$, that is, $%
f_{A}(xyx^{-1})=f_{A}(y)$ for all $x$,$y\in A$,

\item $f_{A}(xyx^{-1})\supseteq f_{A}(y)$ for all $x$,$y\in A$,

\item $f_{A}(xyx^{-1})\subseteq f_{A}(y)$ for all $x$,$y\in A$.
\end{enumerate}
\end{cor}

\begin{proof}
See \cite[Teorem 8]{cag-11}.
\end{proof}


\begin{cor}
\label{C 20} If $G$ is an Abelian group, then any soft int-group in $G$ is a
normal soft int-group.
\end{cor}

\begin{cor}
If $f_{A}\in S_{G}(U)$ and $A\vartriangleleft G,$ then \bigskip $f_{A}\in
NS_{G}(U)$.
\end{cor}


\begin{cor}
\label{C 30} Let $f_{G}\in S_{G}(U).$ Then, $f_{\widetilde{G}}$ and $%
f_{\left( e_{f_{G}}\right) (f_{G}(e))}$ are normal soft int-groups.
\end{cor}


\begin{cor}
\label{C 35} Let $A_{i}\leq G$ for all $i\in I$ and $\{f_{A_{i}}:i\in I\}$
be a family of normal soft int-groups in $G$. Then, $\widetilde{%
\bigcap\limits_{_{i\in I}}}f_{A_{i}}\in NS_{G}(U)$.
\end{cor}


\begin{thrm}
\label{C 90} Let $A\subseteq G\ $and $f_{A}\in S_{G}(U)$. Then, $f_{A}$ is
normal soft int-group (Abelian) if and only if $f_{A}\left( \left[ x,y\right]
\right) =f_{A}\left( e\right) $ for all $x$,$y\in G$, where $\left[ x,y%
\right] =x^{-1}y^{-1}xy$ is commutator of $x$ and $y$.
\end{thrm}

\begin{proof}
Let $f_{A}\left( \left[ x,y\right] \right) =f_{A}\left(
e\right) $ for all $x$,$y\in G.$ Then,%
\begin{eqnarray*}
f_{A}\left( xy\right)  &=&f_{A}\left( \left( yx\right) \left(
yx\right)
^{-1}\left( xy\right) \right)  \\
&\supseteq &f_{A}\left( yx\right) \cap f_{A}\left( x^{-1}y^{-1}xy\right)  \\
&=&f_{A}\left( yx\right) \cap f_{A}\left( e\right)  \\
&=&f_{A}\left( yx\right)
\end{eqnarray*}%
and%
\begin{eqnarray*}
f_{A}\left( yx\right)  &=&f_{A}\left( \left( xy\right) \left(
xy\right)
^{-1}\left( yx\right) \right)  \\
&\supseteq &f_{A}\left( xy\right) \cap f_{A}\left( y^{-1}x^{-1}yx\right)  \\
&=&f_{A}\left( xy\right) \cap f_{A}\left( e\right)  \\
&=&f_{A}\left( xy\right)
\end{eqnarray*}%
so $f_{A}\left( xy\right) =f_{A}\left( yx\right) $.

The converse is obvious.
\end{proof}


\begin{thrm}
\label{C 95} \cite{isaacs-09} Let $N\vartriangleleft G$. Then, $G/N$ is
Abelian if and only if $G%
{\acute{}}%
\subseteq N$, where $G%
{\acute{}}%
$ is the commutator subgroup of the group $G.$
\end{thrm}



\begin{thrm}
\label{C 100} Let $f_{A}\in NS_{G}(U).$ Then, the quotient group $%
G/e_{f_{A}} $ is normal.
\end{thrm}

\begin{proof}
Assume that $f_{A}\in NS_{G}(U)$. Then we have $e_{f_{A}}\vartriangleleft G$%
\ by Corollary \ref{C 30} and $e_{f_{A}}$\ contains commutator subgroup $G%
{\acute{}}%
$ of $G$. So $G/e_{f_{A}}$ is Abelian by Theorem \ref{C 95}, so it
is normal.
\end{proof}


If $H$ and $K$ are two subgroups of a group $G$, then the set $\left[ H,K%
\right] $\ defined as $\left[ H,K\right] =\left\{ \left[ h,k\right] :h\in H%
\text{ and }k\in K\right\} $ is a subgroup of $G$. In addition, we know from
group theory that $H\vartriangleleft G$ if and only if $\left[ H,G\right]
\leq H$ \cite[p.169]{dummit-04}$.$ Analogues to this fact we write the
following theorem.


\begin{thrm}
\label{C 110} Let $A\subseteq G$. Then, a soft int-group $f_{A}$ is a normal
soft int-group if and only if
\begin{equation}
f_{A}\left( \left[ x,y\right] \right) \supseteq f_{A}(x)  \label{C 1100}
\end{equation}%
for all $x,y\in G.$
\end{thrm}

\begin{proof}
Assume that $f_{A}\in NS_{G}(U).$ Let $x$,$y\in G$%
\ then%
\begin{eqnarray*}
f_{A}\left( \left[ x,y\right] \right)  &=&f_{A}\left(
x^{-1}y^{-1}xy\right)
\\
&\supseteq &f_{A}\left( x^{-1}\right) \cap f_{A}\left( y^{-1}xy\right)  \\
&=&f_{A}\left( x\right) \cap f_{A}\left( x\right)  \\
&=&f_{A}\left( x\right) .
\end{eqnarray*}

Conversely, suppose that $f_{A}$\ satisfies (\ref{C 1100}). Then for all $x$,%
$y\in G$, we have

\begin{eqnarray*}
f_{A}\left( x^{-1}yx\right) &=&f_{A}\left( yy^{-1}x^{-1}yx\right) \\
&\supseteq &f_{A}\left( y\right) \cap f_{A}\left( \left[ y,x\right]
\right)
\\
&=&f_{A}\left( y\right)
\end{eqnarray*}%
by assumption, so $f_{A}$ is a normal soft int-group by Corollary \ref{C 15}
\end{proof}


Now, we give some properties of normal soft int-groups including $\alpha $%
-inclusion.

Let $f_{G}^{\alpha _{i}}$ be level subgroups of $G$, where $\alpha _{i}\in $%
Im($f_{G}$), for any $i\in I$ and $\alpha _{0}\supseteq \alpha _{1}\supseteq
\cdots \supseteq \alpha _{r}.$ We know that for any $\alpha _{i}$, $\alpha
_{j}$ in Im($f_{G}$),\ $\alpha _{i}\subseteq \alpha _{j}$ implies $%
f_{G}^{\alpha _{i}}\supseteq f_{G}^{\alpha _{j}}$ by Corollary \ref{B 300}$.$%
\ So any soft int-group on a finite group $G$ gives a chain with subgroups
of $G$ as;%
\begin{equation}
e_{f_{G}}=f_{G}^{\alpha _{0}}\subseteq f_{G}^{\alpha _{1}}\subseteq \cdots
\subseteq f_{G}^{\alpha _{r}}=G.  \label{C 1102}
\end{equation}

We denote this chain of level subgroups by $L_{f}(G)$. It is clear that, all
subgroups of a group $G$ need not form a chain. It follows that not all
subgroups are level subgroups of a soft int-group. 

\begin{definition}
Let $f_{G}\in S_{G}(U)$. Then, $f_{G}$ is called soft level normal subgroup
if and only if the number of the level subgroups are finite and $%
f_{G}^{\alpha _{i}}\vartriangleleft f_{G}^{\alpha _{i+1}}$ for each level
subgroups in $L_{f}(G)$.
\end{definition}



\begin{thrm}
\label{C 190} Let $G$ be a finite group. Then, any normal soft int-group in $%
G$ is a soft level normal subgroup of $G$.
\end{thrm}

\begin{proof}
Assume that $f_{G}\in NS_{G}(U)$ and let
$f_{G}^{\alpha _{i}}$ be as in (\ref{C 1102}). So, for any $x\in
f_{G}^{\alpha _{i}}$,
\begin{equation}
f_{G}\left( y^{-1}xy\right) =f_{G}\left( x\right)  \label{C 1900}
\end{equation}%
for all $y\in G.$ Hence (\ref{C 1900})$\ $is true for all $y\in
f_{G}^{\alpha _{i+1}}$, as well. Consequently $f_{G}^{\alpha
_{i}}\vartriangleleft f_{G}^{\alpha _{i+1}}.$
\end{proof}


\begin{exam}
\label{C 200}Every soft int-group of an Abelian finite group is a soft level
normal subgroup.
\end{exam}


\begin{exam}
\label{C 210}A soft int-group of a finite group may not be a soft level
normal subgroup.

Let $A=\left\{ e,v\right\} \subset G=D_{3}$, where $D_{3}=%
\{e,u,u^{2},v,vu,vu^{2}\}$ is the Dihedral group such that $u^{3}=v^{2}=e$
and $uv=vu^{2}$. It is clear that $\left\{ e\right\} \leq A\leq D_{3}$ is a
chain of subgroups. It is easy to see that $f_{\widetilde{A}}$ is a soft
int-group in $D_{3}.$ However, since%
\begin{equation*}
\phi =f_{\widetilde{A}}(vu)=f_{\widetilde{A}}(vu\left( uu^{-1}\right) )=f_{%
\widetilde{A}}(vu^{2}u^{-1})=f_{\widetilde{A}}(uvu^{-1})\neq f_{\widetilde{A}%
}(v)=U,
\end{equation*}%
$f_{\widetilde{A}}$ is not a normal soft int-group, so it is not a soft
level normal subgroup of $G$.

But if we choose the subgroup chain $\left\{ e\right\} \leq \left\{
e,u,u^{2}\right\} \leq D_{3}$, we observe that all element of chain are soft
level normal subgroup.
\end{exam}



\begin{thrm}
\label{C 220}Let $f_{G}\in S(U)$ and $\alpha \in P(U)$. Then, $f_{G}\in
NS_{G}(U)$ if and only if $f_{G}^{\alpha }$ is a normal subgroup of $G,$
whenever it is nonempty.
\end{thrm}


\begin{proof}
By Theorem \ref{B 367} we know that if $f_{G}\in S_{G}(U)$ then $%
f_{G}^{\alpha }\leq G$ for any nonempty $f_{G}^{\alpha }$. Now we
assume that $f_{G}\in NS_{G}(U).$\ Let $x\in G$, then for any $y\in
f_{G}^{\alpha }$ we have $f_{G}(xyx^{-1})=f_{G}(y)\supseteq \alpha
.$

Thus, $xyx^{-1}\in f_{G}^{\alpha }$ and hence $f_{G}^{\alpha
}\vartriangleleft G.$

Conversely, by Theorem \ref{B 367} we know that if
$f_{G}^{\alpha }\leq G$, for any nonempty $f_{G}^{\alpha }$, then
$f_{G}\in S_{G}(U).$ Now, suppose $f_{G}^{\alpha }\vartriangleleft
G.$\ Let $x$,$y\in G$ and $\ \alpha
=f_{G}(y).$ Then $y\in f_{G}^{\alpha }$ and so $xyx^{-1}\in f_{G}^{\alpha }$%
, since $f_{G}^{\alpha }$ is normal subgroup of $G$.

Hence $f_{G}(xyx^{-1})\supseteq \alpha =f_{G}(y)$ and so $f_{G}\in
NS_{G}(U)$ by Corollary \ref{C 15}.
\end{proof}


\begin{cor}
\label{C 221}Let $f_{G}\in NS_{G}(U).$\ Then, $e$-set ($e_{f_{G}}$) and
support of $f_{G}$\ ($f_{G}^{\ast }$) are normal subgroups of $G$.
\end{cor}

\begin{proof}
Direct by Theorem \ref{C 220}.
\end{proof}

In group theory, a Dedekind group is a group $G$ such that every subgroup of
$G$ is normal. All Abelian groups are Dedekind groups. A non-Abelian
Dedekind group is called a Hamiltonian group.


\begin{exam}
\label{C 224}The quaternion group defined as;
\begin{equation*}
Q=\left\{ -1,i,j,k:\left( -1\right) ^{2}=1,\text{ }i^{2}=j^{2}=k^{2}=ijk=-1%
\right\}
\end{equation*}%
is\ Hamiltonian, where\ 1 is the identity element and -1 commutes with the
other elements of the group.
\end{exam}


\begin{thrm}
\label{C 226}$G$ is a Dedekind group if and only if every soft int-group in $%
G$ is a normal soft int-group.
\end{thrm}

\begin{proof}
Let $G$ be a Dedekind group and $f_{A}\in S_{G}(U).$ We need to
show $f_{A}(y)=f_{A}(xyx^{-1})$ for all $y\in A$ and $x\in G.$ Let $%
f_{A}(y)=\alpha _{1}$ and $f_{A}(xyx^{-1})=\alpha _{2}.$ Then
$f_{A}^{\alpha _{1}}$ and $f_{A}^{\alpha _{2}}$\ are subgroups of
$G$ by Theorem \ref{B 367} and so are normal subgroups of $G$, since
$G$ is a Dedekind group. If $y\in f_{A}^{\alpha _{1}}$ then
$xyx^{-1}\in f_{A}^{\alpha _{1}}$ so $\alpha
_{2}=f_{A}(xyx^{-1})\supseteq \alpha _{1}.$ On the other hand
$xyx^{-1}\in f_{A}^{\alpha _2}$ and since $f_{A}^{\alpha _{2}}$ is
normal, $y\in
f_{A}^{\alpha _2}$ so $\alpha _{1}=f_{A}(y)\supseteq \alpha _{2}.$ Thus $%
\alpha _{1}=\alpha _{2}.$

Conversely, suppose that every soft int-group in $G$ is a normal
soft int-group. Then, by Theorem \ref{C 220} any subgroup $A$ of a
group $G$ can be regarded as a level subgroup of some soft int-group
$f_{A}$ in $G.$ Since every soft int-group in $G$ is a normal soft
int-group, then $A$ is normal subgroup of $G.$ Thus $G$ is a
Dedekind group.
\end{proof}


\begin{lem}
\label{C 227} Let $A$,$B\subseteq G.$ If $f_{A}\in NS_{G}(U) $, then for any
soft set $f_{B}$ in $G$, $f_{A}\ast f_{B}=f_{B}\ast f_{A}$.
\end{lem}

\begin{proof}
For all $x\in G$, we have%
\begin{equation*}
\left( f_{A}\ast f_{B}\right) \left( x\right) =\bigcup \left\{
f_{A}(u)\cap f_{B}(v):uv=x,\text{ }u,v\in G\right\}
\end{equation*}%
and since $f_{A}$ is normal soft int-group and $uv=x$ implies $u=xv^{-1}$, then%
\begin{eqnarray*}
\left( f_{A}\ast f_{B}\right) \left( x\right)  &=&\bigcup \left\{
f_{A}(xv^{-1})\cap f_{B}(v):(xv^{-1})v=x,\text{ }v\in G\right\}  \\
&=&\bigcup \left\{ f_{B}(v)\cap f_{A}(v^{-1}x):v(v^{-1}x)=x,\text{
}v\in
G\right\}  \\
&=&\left( f_{B}\ast f_{A}\right) \left( x\right) .
\end{eqnarray*}
\end{proof}


\begin{thrm}
\label{C 228}If $f_{A}\in NS_{G}(U)$ and $f_{B}\in S_{G}(U)$, then $\left(
f_{A}\ast f_{B}\right) \in S_{G}(U)$.
\end{thrm}

\begin{proof}
Firstly,
\begin{eqnarray*}
\left( f_{A}\ast f_{B}\right) \ast \left( f_{A}\ast f_{B}\right)
&=&f_{A}\ast \left( f_{B}\ast f_{A}\right) \ast f_{B} \\
&=&f_{A}\ast \left( f_{A}\ast f_{B}\right) \ast f_{B}\text{ (by
Lemma \ref{C
227})} \\
&=&\left( f_{A}\ast f_{A}\right) \ast \left( f_{B}\ast f_{B}\right)  \\
&\subseteq &f_{A}\ast f_{B}\text{ \ \ \ \ \ \ \ \ \ (by Theorem
\ref{B 480})}
\end{eqnarray*}%
Secondly, for all $x\in G$%
\begin{eqnarray*}
\left( f_{A}\ast f_{B}\right) \left( x^{-1}\right)  &=&\bigcup
\left\{
f_{A}(u)\cap f_{B}(v):uv=x^{-1},\text{ }u,v\in G\right\}  \\
&=&\bigcup \left\{ f_{A}\left( (u^{-1})^{-1}\right) \cap f_{B}\left(
(v^{-1})^{-1}\right) :v^{-1}u^{-1}=x,\text{ }u,v\in G\right\}  \\
&=&\bigcup \left\{ f_{B}(v^{-1})\cap f_{A}(u^{-1}):v^{-1}u^{-1}=x,\text{ }%
u,v\in G\right\}  \\
&=&\left( f_{B}\ast f_{A}\right) \left( x\right)  \\
&=&\left( f_{A}\ast f_{B}\right) \left( x\right) \text{ ( by Lemma
\ref{C 227}).}
\end{eqnarray*}

Hence $\left( f_{A}\ast f_{B}\right) \in S_{G}(U)$ by Theorem \ref{B
480}.
\end{proof}


\begin{cor}
\label{C 229}If\bigskip\ $f_{A}$,$f_{B}\in NS_{G}(U)$, then $\left(
f_{A}\ast f_{B}\right) \in NS_{G}(U).$
\end{cor}

\begin{proof}
$\left( \bigskip f_{A}\ast f_{B}\right) \in S_{G}(U)$ by Theorem \ref{C 228}%
, so we should verify that $f_{A}\ast f_{B}$ is a normal
soft int-group. For any $%
x\in G$, we have%
\begin{eqnarray*}
\left( f_{A}\ast f_{B}\right) \left( x\right)  &=&\bigcup \left\{
f_{A}(u)\cap f_{B}(v):uv=x,\text{ }u,v\in G\right\}  \\
&=&\bigcup \left\{ f_{A}(w^{-1}uw)\cap f_{B}(w^{-1}vw):\left(
w^{-1}uw\right) \left( w^{-1}vw\right) =w^{-1}xw,\text{ }u,v\in G\right\}  \\
&=&\left( f_{A}\ast f_{B}\right) \left( w^{-1}xw\right)
\end{eqnarray*}%
for all $w$,$x\in G.$ Hence $\left( f_{A}\ast f_{B}\right) \in
NS_{G}(U)$.
\end{proof}


\begin{remark}
\label{C 240}$\left( NS_{G}(U),\ast \right) $ is a commutative idempotent
semigroup, because

\begin{enumerate}
\item $NS_{G}(U)$ is closed under operation $"\ast "$, (by Corollary \ref{C
229})

\item $\left( NS_{G}(U),\ast \right) $ is commutative, (by Theorem \ref{B
490})

\item $\left( NS_{G}(U),\ast \right) $ is associative, (by Theorem \ref{B
400})

\item $\left( NS_{G}(U),\ast \right) $ is idempotent, i.e., $f_{G}\ast
f_{G}=f_{G}$.
\end{enumerate}
\end{remark}


\begin{definition}
\label{C 243}Let $f_{A}$,$f_{B}\in S_{G}(U).$ Then, $f_{A}$ and $f_{B}$ are
called conjugate soft int-groups (with respect to $u$), if there exists $%
u\in G$ such that, $f_{A}(x)=f_{B}(uxu^{-1})$, for all $x\in G$ and we
denote $f_{A}=f_{B^{u}}$, where $f_{B^{u}}(x)=f_{B}(uxu^{-1})$, for all $%
x\in G.$
\end{definition}


\begin{thrm}
\label{C 246}A soft int-group $f_{A}$ is a normal soft int-group in $G$ if
and only if $f_{A}$ is constant on each conjugate class of $G$, that is, $%
f_{A^{u}}=f_{A}$ for all $u\in G.$
\end{thrm}

\begin{proof}
Suppose $f_{A}\in NS_{G}(U).$ Then,
\begin{equation*}
f_{A}\left( uxu^{-1}\right) =f_{A}(xuu^{-1})=f_{A}(x)\text{, for all
}x,u\in G.
\end{equation*}%
{}

Conversely suppose that $f_{A}$ is constant on each conjugate class
of $G$. Then
\begin{equation*}
f_{A}\left( xu\right) =f_{A}(xuxx^{-1})=f_{A}(x\left( ux\right)
x^{-1})=f_{A}(ux)
\end{equation*}%
\ for all $x$,$u\in G$, so $f_{A}\in NS_{G}(U)$.
\end{proof}


\begin{definition}
\label{C 260}If $f_{G}\in S_{G}(U)$, then the set defined as$\ $%
\begin{equation*}
N(f_{G})=\{x\in G:f_{G}(xy)=f_{G}(yx)\}
\end{equation*}%
for all $y\in G,$ is called normalizer of $f_{G}$ in $G$.
\end{definition}

Clearly, for all $f_{G}\in S_{G}(U)$, the unit element of a group $G $ is in
$N(f_{G})$ and if $G$ is Abelian then $N(f_{G})=G.$

\begin{cor}
If $x\in N(f_{G})$, then $x^{-1}\in N(f_{G}).$
\end{cor}

\begin{proof}
Let $x\in N(f_{G}).$ Then, for all $u\in G$,
\begin{equation*}
f_{G}(x^{-1}u)=f_{G}(x^{-1}u\left( xx^{-1}\right)
)=f_{G}(x^{-1}\left( ux\right) x^{-1})=f_{G}(x^{-1}\left( xu\right)
x^{-1})=f_{G}(ux^{-1})
\end{equation*}%
so $x^{-1}\in N(f_{G}).$\bigskip
\end{proof}


\begin{cor}
If $f_{G}\in S_{G}(U)$, then
\begin{equation}
N(f_{G})=\left\{ u\in G:f_{G^{u}}=f_{G}\right\} .  \label{C 2650}
\end{equation}
\end{cor}


\begin{thrm}
\label{C 270}Let $f_{G}\in S_{G}(U)$. Then,

\begin{enumerate}
\item $N(f_{G})\leq G$,

\item The restriction of $f_{G}$ to $N\left( f_{G}\right) $ is a normal soft
int-group, that is $f_{G}|_{N(f_{G})}\in NS_{G}(U)$,

\item $f_{G}\in NS_{G}(U)$ if and only if $N(f_{G})=G.$
\end{enumerate}
\end{thrm}

\begin{proof}
Let $f_{G}\in S_{G}(U).$

\begin{enumerate}
\item Obviously $N(f_{G})\neq \phi $, since $e\in N(f_{G}).$\ Let $x,y\in
N(f_{G}).$ Then, for all $u\in G$,
\begin{eqnarray*}
f_{G}(\left( xy^{-1}\right) u) &=&f_{G}(x\left( y^{-1}u\right) ) \\
&=&f_{G}(x\left( uy^{-1}\right) )\text{ (since }y^{-1}\in N(f_{G})\text{)} \\
&=&f_{G}(\left( xu\right) y^{-1}) \\
&=&f_{G}(\left( ux\right) y^{-1}) \\
&=&f_{G}(u\left( xy^{-1}\right) )
\end{eqnarray*}%
so $xy^{-1}\in N(f_{G}).$ Hence $N(f_{G})$ is a subgroup of $G$.

\item $f_{G}|_{N(f_{G})}$ is a soft int-group by Theorem \ref{B 210}.
Since $N(f_{G})$ is Abelian, $f_{G}|_{N(f_{G})}(xy)=f_{G}\mid
_{N(f_{G})}(yx)$, for all $x$,$y\in N(f_{G})$. Hence,
$f_{G}|_{N(f_{G})}$ is a normal soft int-group.

\item Suppose that $f_{G}$ is a normal soft int-group and $u\in G.$ Then for any $%
g\in G$\ we have
\begin{eqnarray*}
f_{G^{u}}(g) &=&f_{G}(u^{-1}gu) \\
&=&f_{G}\left( u\left( u^{-1}g\right) \right) \text{(by
assumption)} \\
&=&f_{G}\left( g\right) .
\end{eqnarray*}%
So $f_{G^{u}}=f_{G}$\ and hence\ $u\in N(f_{G})$ by (\ref{C 2650}),
which
implies $G\subseteq N(f_{G}).$ Since $N(f_{G})\leq G$ then $%
N(f_{G})\subseteq G$ and so $N(f_{G})=G.$
\end{enumerate}

Conversely, let $N(f_{G})=G$ and $x$,$y\in G.$ Then we have

\begin{equation}
f_{G}\left( xy\right) =f_{G}\left( xyxx^{-1}\right) =f_{G}\left(
\left( x^{-1}\right) ^{-1}\left( yx\right) x^{-1}\right)
=f_{G^{x^{-1}}}(yx). \label{C 270-1}
\end{equation}%
On the other hand, since $N(f_{G})=G$, for any $x\in G=N(f_{G})$ we have $%
x^{-1}\in N(f_{G})$, and so
\begin{equation}
f_{G^{x^{-1}}}(yx)=f_{G}\left( yx\right)   \label{C 270-2}
\end{equation}%
by the definition of normalizer. Thus $f_{G}\left( xy\right)
=f_{G}\left( yx\right) $, for all $x$,$y\in G$ from (\ref{C 270-1})
and (\ref{C 270-2}).

Hence, $f_{G}$ is a normal soft int-group.
\end{proof}


\begin{thrm}
\label{C 290}Let $G$ be a finite group and $f_{G}\in S_{G}(U)$ such that $%
f_{G}^{\ast }\neq \phi $. Then, the number of distinct conjugate classes of $%
f_{G}$ is equal to the index of $N(f_{G})$ in $G$.
\end{thrm}

\begin{proof}
Since $N(f_{G})\leq G$,\ $G$ can be written as a union of cosets of
$N(f_{G})
$, as%
\begin{equation*}
G=x_{1}N(f_{G})\cup x_{2}N(f_{G})\cup \ldots \cup x_{k}N(f_{G}),
\end{equation*}%
where $k$ is the number of distinct cosets, that is $k=\left\vert
G:N(f_{G})\right\vert .$ Let $x\in N(f_{G})$ and choose $i$ such
that $1\leq
i\leq k$. Then for any $g\in G$,%
\begin{eqnarray*}
f_{G^{x_{i}x}}\left( g\right)  &=&f_{G}\left( \left( x_{i}x\right)
^{-1}g\left( x_{i}x\right) \right)  \\
&=&f_{G}\left( x^{-1}\left( x_{i}^{-1}gx_{i}\right) x\right)  \\
&=&f_{G^{x}}\left( x_{i}^{-1}gx_{i}\right)  \\
&=&f_{G}\left( x_{i}^{-1}gx_{i}\right) \text{ (since }x\in N(f_{G})) \\
&=&f_{G^{x_{i}}}\left( g\right) .
\end{eqnarray*}%
Thus, we have $f_{G^{x_{i}x}}\left( g\right) =f_{G^{x_{i}}}\left( g\right) $%
, for all $x\in N(f_{G})$ and $1\leq i\leq k$.

So any two elements in $G$, which lie in the same coset
$x_{i}N(f_{G})$\ give rise to the same conjugate $f_{G^{x_{i}}}$ of
$f_{G}.$ Now we show that
two distinct cosets give two distinct conjugates of $f_{G}.$ Suppose that $%
f_{G^{x_{i}}}=f_{G^{xj}}$, where $i\neq j$ and $1\leq i$,$j\leq k.$
Thus,
for all $g\in G$,%
\begin{eqnarray*}
f_{G^{x_{i}}} &=&f_{G^{xj}} \\
&\Leftrightarrow &f_{G^{x_{i}}}\left( g\right) =f_{G^{xj}}\left( g\right)  \\
&\Leftrightarrow &f_{G}\left( x_{i}^{-1}gx_{i}\right) =f_{G}\left(
x_{j}^{-1}gx_{j}\right) .
\end{eqnarray*}%
If we choose $g=x_{j}tx_{j}^{-1}$, it follows that%
\begin{equation*}
f_{G}\left( x_{i}^{-1}\left( x_{j}tx_{j}^{-1}\right) x_{i}\right)
=f_{G}\left( x_{j}^{-1}\left( x_{j}tx_{j}^{-1}\right) x_{j}\right)
\end{equation*}%
\begin{eqnarray*}
&\Rightarrow &f_{G}\left( \left( x_{j}^{-1}x_{i}\right) ^{-1}t\left(
x_{j}^{-1}x_{i}\right) \right) =f_{G}\left( t\right) \text{ for all
}t\in G
\\
&\Rightarrow &f_{G^{x_{j}^{-1}x_{i}}}\left( t\right) =f_{G}\left(
t\right)
\text{ for all }t\in G \\
&\Rightarrow &x_{j}^{-1}x_{i}\in N(f_{G}) \\
&\Rightarrow &x_{i}N(f_{G})=x_{j}N(f_{G}).
\end{eqnarray*}%
However, if $i\neq j$, this is not possible when we consider the
decomposition of $G$ as a union of cosets of $N(f_{G}).$ Hence the
number of distinct conjugates of $f_{G}$ is equal to $\left\vert
G:N(f_{G})\right\vert $.
\end{proof}


\begin{thrm}
\label{C 300}Let $f_{A}\in S_{G}(U)$ and $f_{A^{u}}$ be as in Definition (%
\ref{C 243}). Then,

\begin{enumerate}
\item $\underset{u\in G}{\widetilde{\bigcap }}f_{A^{u}}$ is a normal soft
int-group,

\item $\underset{u\in G}{\widetilde{\bigcap }}f_{A^{u}}$ is the largest
normal soft int-group in $G$, contained in $f_{A}$.
\end{enumerate}
\end{thrm}

\begin{proof}
Let $f_{A}\in S_{G}(U)$. Then,

\begin{enumerate}
\item $\underset{u\in G}{\widetilde{\bigcap }}f_{A^{u}}$ is a soft int-group, since $f_{A^{u}}$\ are soft int-groups, for all $u\in
G$, by
Theorem \ref{B 220}$.$ Now, for all $x$,$y\in G$%
\begin{eqnarray*}
\underset{u\in G}{\bigcap }f_{A^{u}}\left( xyx^{-1}\right)  &=&\underset{%
u\in G}{\bigcap }f_{A}\left( u\left( xyx^{-1}\right) u^{-1}\right)  \\
&=&\underset{u\in G}{\bigcap }f_{A}\left( \left( ux\right) y\left(
ux\right)
^{-1}\right)  \\
&=&\underset{u\in G}{\bigcap }f_{A^{ux}}\left( y\right)  \\
&=&\underset{u\in G}{\bigcap }f_{A^{u}}\left( y\right) \text{ }.
\end{eqnarray*}%
since $f_{A^{u}}$\ and $f_{A^{ux}}$\ are in the same conjugate class of $%
f_{A}$, for all $x\in G.$ Thus, $\underset{u\in G}{\widetilde{\bigcap }}%
f_{A^{u}}$ is a normal soft int-group, by Theorem \ref{C 246}.

\item Let $f_{B}$ be a normal soft int-group\ satisfying $f_{B}\subseteq
f_{A}.$ Then $f_{B}=f_{B^{u}}\widetilde{\subseteq }f_{A^{u}}$ for
all $u\in G
$, by assumption. Thus, $f_{B}\subseteq \underset{u\in G}{%
\widetilde{\bigcap }}f_{A^{u}}.$ Therefore $\underset{u\in G}{\widetilde{%
\bigcap }}f_{A^{u}}$\ is the largest normal soft int-group in $G$,
contained in $f_{A}$.
\end{enumerate}
\end{proof}


Now, we introduce the notion of coset.

Let $f_{G}\in S_{G}(U)$ and $a\in G$. Then the soft subsets $f_{a\left(
f_{G}(e)\right) }\ast f_{G}$ and $f_{G}\ast f_{a\left( f_{G}(e)\right) }$
are referred to as the left coset and right coset of $f_{G}$ with respect to
$a$.

From Corollary \ref{B 420}, we have that $\left( f_{a\left( f_{G}(e)\right)
}\ast f_{G}\right) \left( y\right) =f_{G}\left( a^{-1}y\right) $ and $\left(
f_{G}\ast f_{a\left( f_{G}(e)\right) }\right) \left( y\right) =f_{G}\left(
ya^{-1}\right) $, since $f_{G}(e)\supseteq f_{G}(x)$, for all $x\in G$. On
the base of these facts, the following definition is given:


\begin{definition}
\label{C 340}\cite{cag-11} Let $f_{G}\in S_{G}(U)$ and $a\in G$. Then, soft
left coset of $f_{G}$, denoted by $af_{G}$, is defined by the approximation
function $(af_{G})(x)=f_{G}(a^{-1}x)$ for all $x\in G$.
\end{definition}

\bigskip Similarly, right coset of $f_{G}$ can be defined by the
approximation function $(f_{G}a)(x)=f_{G}(xa^{-1})$\ for all $x\in G$ and
denoted by $f_{G}a$.

If $f_{G}\in NS_{G}(U)$, then soft the left coset is equal to the soft right
coset. Thus in this case, we call only soft coset and denote by $af_{G}$.


Definition of coset above is analogues to definition of classical algebra as
follows:

Let $H\leq G$ and $f_{H}$ be the characteristic function of $H$, that is%
\begin{equation*}
f_{H}(x)=\left\{
\begin{array}{ll}
U & for\text{ }x\in H \\
\phi & for\text{ }x\in G\backslash H%
\end{array}%
\right. .
\end{equation*}%
It is well known that for any $a\in G$, $aG=G$.

Now if $g\in H$, then $ag\in aH$, so
\begin{equation*}
af_{H}(ag)=f_{H}\left( a^{-1}ag\right) =f_{H}\left( g\right) =U.
\end{equation*}

If $g\notin H$, then $ag\notin aH$, and so
\begin{equation*}
af_{H}(ag)=f_{H}\left( a^{-1}ag\right) =f_{H}\left( g\right) =\phi .
\end{equation*}

Thus, it follows that $af_{H}$ is a function on $G$, such that%
\begin{equation*}
af_{H}(x)=\left\{
\begin{array}{ll}
U & for\text{ }x\in aH \\
\phi & for\text{ }x\in G\backslash \left( aH\right)%
\end{array}%
\right. .
\end{equation*}

This shows that $af_{H}$ is the characteristic function of $aH$.



\begin{prop}
\label{C 345}Let $f_{G}\in S_{G}(U)$. Then, there is a one-to-one
correspondence between the set of right cosets and the set of left cosets of
$f_{G}$\ in $G$.
\end{prop}



\begin{thrm}
\label{C 355}Let $f_{G}\in NS_{G}(U).$Then, for any $a\in G$%
\begin{equation*}
af_{G}(ga)=af_{G}\left( ag\right) =f_{G}\left( g\right) ,\text{ for all }%
g\in G.
\end{equation*}
\end{thrm}

\begin{proof}
Let $f_{G}\in NS_{G}(U)$. Then, for any $a\in G$%
\begin{equation*}
af_{G}(ga)=f_{G}\left( a^{-1}ga\right) =f_{G}\left( g\right)
\end{equation*}%
since $f_{G}\in NS_{G}(U).$ The other part is similar.
\end{proof}


\begin{thrm}
\label{C 360}\cite{cag-11} Let $f_{G}\in S_{G}(U)$. Then, $%
af_{G}=bf_{G}\Leftrightarrow ae_{f_{G}}=be_{f_{G}}$ for all $a$,$b\in G$.
\end{thrm}


\begin{thrm}
\label{C 370}\cite{cag-11} Let $f_{G}\in NS_{G}(U)$. If $af_{G}=bf_{G}$,
then $f_{G}(a)=f_{G}(b)$ for any $a$,$b\in G$.
\end{thrm}


Now, we introduce the notion of quotient groups.

\begin{thrm}
\label{C 380}Let $f_{G}\in NS_{G}(U)$ and define a set $G/f_{G}=\{xf_{G}:x%
\in G\}$. Then, the following assertions hold:

\begin{enumerate}
\item $\left( xf_{G}\right) \ast \left( yf_{G}\right) =\left( xyf_{G}\right)
$, for all $x$,$y\in G$,

\item $(G/f_{G},\ast )$ is a group. Moreover, if $G$ is Abelian then so is $%
G/f_{G}.$
\end{enumerate}
\end{thrm}

\begin{proof}
Let $f_{G}\in NS_{G}(U)$.

\begin{enumerate}
\item For all $x$,$y\in G$,
\begin{eqnarray*}
\left( xf_{G}\right) \ast \left( yf_{G}\right)  &=&\left( f_{x\left(
f_{G}(e)\right) }\ast f_{G}\right) \ast \left( f_{y\left(
f_{G}(e)\right)
}\ast f_{G}\right)  \\
&=&f_{x\left( f_{G}(e)\right) }\ast \left( f_{G}\ast f_{y\left(
f_{G}(e)\right) }\right) \ast f_{G} \\
&=&f_{x\left( f_{G}(e)\right) }\ast \left( f_{y\left(
f_{G}(e)\right) }\ast
f_{G}\right) \ast f_{G}\text{ \ } \\
&=&\left( f_{x\left( f_{G}(e)\right) }\ast f_{y\left(
f_{G}(e)\right)
}\right) \ast \left( f_{G}\ast f_{G}\right)  \\
&=&\left( f_{x\left( f_{G}(e)\right) }\ast f_{y\left(
f_{G}(e)\right)
}\right) \ast f_{G}\text{ (by Remark \ref{C 240})} \\
&=&f_{\left( xy\right) \left( f_{G}(e)\right) }\ast f_{G}\text{ (by Theorem %
\ref{B 380})} \\
&=&xyf_{G}.
\end{eqnarray*}

\item $(G/f_{G},\ast )$\ is closed under the operation "$\ast $" by part 1
and it is associative by Theorem \ref{B 400}. Now, for all $x\in G$%
\begin{equation*}
f_{G}\ast xf_{G}=ef_{G}\ast xf_{G}=\left( ex\right)
f_{G}=xf_{G}=\left( xe\right) f_{G}=xf_{G}\ast ef_{G}=xf_{G}\ast
f_{G}
\end{equation*}%
so, $f_{G}=ef_{G}$ is an identity element of $G/f_{G}.$ In addition,
for all
$x\in G$%
\begin{equation*}
\left( x^{-1}f_{G}\right) \ast \left( xf_{G}\right) =\left(
x^{-1}x\right) f_{G}=ef_{G}=\left( xx^{-1}\right) f_{G}=\left(
xf_{G}\right) \ast \left( x^{-1}f_{G}\right)
\end{equation*}%
so, $\left( x^{-1}f_{G}\right) $ is the inverse of $\left(
xf_{G}\right) .$ Hence $(G/f_{G},\ast )$\ is a group.
\end{enumerate}

Moreover if $G$ is Abelian then for all $x$,$y\in G$%
\begin{equation*}
xf_{G}\ast yf_{G}=xyf_{G}=yxf_{G}=yf_{G}\ast xf_{G}
\end{equation*}%
so $G/f_{G}$ is Abelian$.$
\end{proof}


\begin{definition}
\label{C 382}Let $f_{G}\in NS_{G}(U)$. Then, the group $G/f_{G}$ defined in
Theorem \ref{C 380} is called the quotient (or factor) group of $G$ relative
to the normal soft int-group $f_{G}$.
\end{definition}


\begin{thrm}
\label{C 383}Let $f_{G}\in NS_{G}(U)$. Then, $G/f_{G}\cong G/e_{f_{G}}$.
\end{thrm}

\begin{proof}
Since $f_{G}\in NS_{G}(U)$, $e_{f_{G}}$ is a normal subgroup of $G$,
by Corollary \ref{C 221} and hence $G/e_{f_{G}}$ is a quotient
group. In
addition, $G/f_{G}$ is a group by Theorem \ref{C 380}. Now, define a map $%
\varphi :G/f_{G}\rightarrow G/e_{f_{G}}$ by setting $\varphi \left(
xf_{G}\right) =xe_{f_{G}}.$ Firstly, $\varphi $ is a homomorphism
since, for
all $xf_{G}$,$yf_{G}\in G/f_{G}$%
\begin{equation*}
\varphi \left( \left( xf_{G}\right) \ast \left( yf_{G}\right)
\right) =\varphi \left( xyf_{G}\right) =\left( xy\right)
e_{f_{G}}=\left( xe_{f_{G}}\right) \left( ye_{f_{G}}\right) =\varphi
\left( xf_{G}\right) \varphi \left( yf_{G}\right) .
\end{equation*}%
On the other hand, $\varphi $ is a bijection by Theorem \ref{C 360}.

Hence $G/f_{G}\cong G/e_{f_{G}}$.
\end{proof}


\begin{thrm}
\label{C 385}Let $f_{G}\in NS_{G}(U)$ and define a soft set $f_{G}^{\left(
\ast \right) }$ on $G/f_{G}$ by, $f_{G}^{\left( \ast \right) }\left(
xf_{G}\right) =f_{G}(x)$, for all $x\in G$. Then, $f_{G}^{\left( \ast
\right) }$ is a normal soft int-group in $G/f_{G}$.
\end{thrm}

\begin{proof}
Firstly, $f_{G}^{\left( \ast \right) }$\ is well defined since, for any $%
xf_{G}$,$yf_{G}\in G/f_{G}$
\begin{eqnarray*}
\left( xf_{G}\right)  &=&\left( yf_{G}\right) \Rightarrow
f_{G}\left(
x\right) =f_{G}\left( y\right) \text{ (by the Theorem \ref{C 370})} \\
&\Rightarrow &f_{G}^{\left( \ast \right) }\left( xf_{G}\right)
=f_{G}^{\left( \ast \right) }\left( yf_{G}\right) .
\end{eqnarray*}%
Secondly, for all $x$,$y\in G$%
\begin{eqnarray*}
f_{G}^{\left( \ast \right) }\left( \left( xf_{G}\right) \ast \left(
yf_{G}\right) \right)  &=&f_{G}^{\left( \ast \right) }\left(
xyf_{G}\right)
\\
&=&f_{G}\left( xy\right)  \\
&\supseteq &f_{G}\left( x\right) \cap f_{G}\left( y\right)  \\
&=&f_{G}^{\left( \ast \right) }\left( xf_{G}\right) \cap
f_{G}^{\left( \ast \right) }\left( yf_{G}\right)
\end{eqnarray*}%
and for all $x\in G$,%
\begin{equation*}
f_{G}^{\left( \ast \right) }\left( \left( xf_{G}\right) ^{-1}\right)
=f_{G}^{\left( \ast \right) }\left( x^{-1}f_{G}\right) =f_{G}\left(
x^{-1}\right) =f_{G}\left( x\right) =f_{G}^{\left( \ast \right)
}\left( xf_{G}\right)
\end{equation*}%
thus $f_{G}^{\left( \ast \right) }\in S_{G/f_{G}}(U).$ Thirdly, for all $x$,$%
y\in G$,
\begin{eqnarray*}
f_{G}^{\left( \ast \right) }\left( \left( xf_{G}\right) \ast \left(
yf_{G}\right) \right)  &=&f_{G}^{\left( \ast \right) }\left(
xyf_{G}\right)
\\
&=&f_{G}\left( xy\right)  \\
&=&f_{G}\left( yx\right)  \\
&=&f_{G}^{\left( \ast \right) }\left( yxf_{G}\right)  \\
&=&f_{G}^{\left( \ast \right) }\left( \left( yf_{G}\right) \ast
\left( xf_{G}\right) \right)
\end{eqnarray*}%
hence $f_{G}^{\left( \ast \right) }\in NS_{G/f_{G}}(U).$
\end{proof}

\bigskip

\begin{thrm}
\label{D 376}Let $f_{A},f_{B}\in S_{H}(U)$ be such that $f_{A}\widetilde{%
\subseteq }f_{B},$ $H$ be a group and $\varphi :G\rightarrow H$ be a
homomorphism. Then $\varphi ^{-1}(f_{A})\widetilde{\subseteq }\varphi
^{-1}(f_{B})$.
\end{thrm}

\begin{proof}
For all $x\in G$
\begin{equation*}
\varphi ^{-1}\left( f_{A}\right) \left( x\right) =\text{ }f_{A}\left(
\varphi \left( x\right) \right) \subseteq \text{ }f_{B}\left( \varphi \left(
x\right) \right) =\varphi ^{-1}\left( f_{B}\right) \left( x\right) .
\end{equation*}%
So, $\varphi ^{-1}(f_{A})\widetilde{\subseteq }\varphi ^{-1}(f_{B})$.
\end{proof}

\begin{thrm}
\label{C 420}Let $f_{G}\in NS_{G}(U)$ and $H$ be a group. If $\varphi $ is
an epimorphism from $G$ onto $H$ then, $\varphi (f_{G})\in NS_{H}(U)$.
\end{thrm}

\begin{proof}
\bigskip We have $\varphi (f_{G})\in S_{H}(U)$ (see \cite{cag-11} Theorem
19). Since $\varphi $ is onto there exist $u$,$v\in G$ such that
$\varphi \left( u\right) =x$ and $\varphi \left( v\right) =y$ for
any $x$,$y\in H.$ Thus for all $x$,$y\in H$,
\begin{eqnarray*}
\varphi (f_{G})\left( xy\right)  &=&\bigcup \left\{ f_{G}(w):w\in G,\text{ }%
\varphi (w)=xy\right\}  \\
&=&\bigcup \left\{ f_{G}(uv):uv\in G,\text{ }\varphi (uv)=xy\right\}  \\
&=&\bigcup \left\{ f_{G}(uv):vu\in G,\text{ }\varphi (u)\text{
}\varphi
(v)=xy\right\} \text{ (since }\varphi \text{ is a homomorphism)} \\
&=&\bigcup \left\{ f_{G}(uv):vu\in G,\text{ }\varphi (u)=x\text{ and }%
\varphi (v)=y\right\}  \\
&=&\bigcup \left\{ f_{G}(vu):vu\in G,\text{ }\varphi (v)\varphi
(u)=yx\right\} \text{ (since $f_{G}\in NS_{G}(U)$}) \\
&=&\bigcup \left\{ f_{G}(vu):vu\in G,\text{ }\varphi (vu)=yx\right\}  \\
&=&\varphi (f_{G})\left( yx\right) .
\end{eqnarray*}%
Hence $\varphi (f_{G})$ is a normal soft int-group.
\end{proof}


\begin{thrm}
\label{C 430}Let $H$ be a group and $f_{H}\in NS_{H}(U)$. If $\varphi $ is a
homomorphism from $G$ into $H$, then $\varphi ^{-1}(f_{H})\in NS_{G}(U)$.
\end{thrm}

\begin{proof}
\bigskip \bigskip We have $\varphi ^{-1}(f_{H})\in S_{G}(U)$ (see \cite%
{cag-11} Theorem 20). For all $x$,$y\in G$%
\begin{eqnarray*}
\varphi ^{-1}(f_{H})\left( xy\right)  &=&f_{H}\left( \varphi \left(
xy\right) \right)  \\
&=&f_{H}\left( \varphi \left( x\right) \varphi \left( y\right) \right)  \\
&=&f_{H}\left( \varphi \left( y\right) \varphi \left( x\right)
\right) \text{
(since $f_{H}\in NS_{H}(U)$}) \\
&=&f_{H}\left( \varphi \left( yx\right) \right)  \\
&=&\varphi ^{-1}(f_{H})\left( yx\right) .
\end{eqnarray*}%
Hence $\varphi ^{-1}(f_{H})$ is normal soft int-group.
\end{proof}


\begin{lem}
\label{B 163}\bigskip Let $\varphi :A\rightarrow B$ be a function. Then, for
all $f_{B}\in S(U),$\ $f_{B}\widetilde{\supseteq }\varphi \left( \varphi
^{-1}(f_{B})\right) .$
\end{lem}

\bigskip In particular, if $\varphi $ is a surjective function, then $%
f_{B}=\varphi \left( \varphi ^{-1}(f_{B})\right) $.

\begin{thrm}
\bigskip \label{D 593}Let $f_{B}\in NS_{H}(U),$ $H$ be a group and $\varphi
:G\rightarrow H$ be a homomorphism. Then, $\varphi \left( \varphi
^{-1}(f_{B})\right) \in NS_{H}(U).$
\end{thrm}

\begin{proof}
Let $f_{B}\in NS_{H}(U).$ Then,
\begin{eqnarray*}
\varphi \left( \varphi ^{-1}(f_{B})\right) \left( xyx^{-1}\right)  &=&\cup
\left\{ \varphi ^{-1}(f_{B})(z):z\in G,\text{ }\varphi \left( z\right)
=xyx^{-1}\right\}  \\
&=&\cup \left\{ f_{B}\left( \varphi (z)\right) :z\in G,\text{ }\varphi
\left( z\right) =xyx^{-1}\right\}  \\
&=&\cup \left\{ f_{B}\left( \varphi (xyx^{-1})\right) \right\}  \\
&=&\cup \left\{ f_{B}\left( \varphi (x)\varphi (y)\varphi (x)^{-1}\right)
\right\}  \\
&\supseteq &\cup \left\{ f_{B}\left( \varphi (y)\right) \right\} \text{ \
(since }f_{B}\in NS_{H}(U)\text{)} \\
&=&\cup \left\{ \varphi ^{-1}(f_{B})(y)\right\}  \\
&=&\varphi \left( \varphi ^{-1}(f_{B})\right) \left( y\right) \text{ \ (by
Definition \ref{B 111})}
\end{eqnarray*}%
for all $x,y\in G$.
\end{proof}

\section{\protect\bigskip Conclusion}

In this paper, we studied on normal soft int-groups\ and investigate
relations with $\alpha $-inclusion and soft product. Then, we define
normalizer, quotient group and give some theorems concerning these concepts.
For future works, it is possible to study on isomorphism theorems and other
concepts of group theory.


\begin{thebibliography}{99}
\bibitem{abo-93} Abou-Zaid, S., On fuzzy subgroups, Fuzzy Sets Syst., 55
(1993) 237-240.

\bibitem{aca-10} Acar, U., Koyuncu, F. and Tanay, B., Soft sets and soft
rings, Comput. Math. Appl. 59, 3458-3463 (2010).

\bibitem{ajm-92} Ajmal, N. and Prajapati, A. S., Fuzzy cosets and fuzzy
normal subgroups, Inform. Sci. 64 (1992) 17-25.

\bibitem{akg-88} Akg\"{u}l, M., Some properties of fuzzy groups, J. Math.
Anal. Appl. 133 (1988) 93-100. 29.

\bibitem{akt-07} Akta\c{s} H. and \c{C}a\u{g}man, N., Soft sets and soft
groups, Inform. Sci. 177, 2726-2735 (2007).

\bibitem{ali-09} Ali, M.I., Feng, F., Liu, X., Min W.K. and Shabir, M., On
some new operations in soft set theory, Comput. Math. Appl. 57, 1547-1553
(2009).

\bibitem{ant-79} Anthony J. M. and Sherwood, H., Fuzzy subgroups redefined,
J. Math. Anal. Appl. 69 (1979) 124-130.

\bibitem{asa-91} Asaad, M., Groups and fuzzy subgroups, Fuzzy Sets Syst. 39
(1991) 323- 328.

\bibitem{bhu-93} Bhutani, K. R., Fuzzy Sets, Fuzzy Relations and Fuzzy
Groups: Some Interrelations, Inform. Sci. 73(1993), 107-115.

\bibitem{cag-11} \c{C}a\u{g}man, N., \c{C}\i tak F. and Akta\c{s}, H., Soft
int-group and its applications to group theory, Neural Comput. and Appl.,
DOI:10.1007/s00521-011-0752-x.

\bibitem{cag-10} \c{C}a\u{g}man N. and Engino\u{g}lu, S., Soft set theory
and uni-int decision making, Eur. J. Oper. Res. 207, 848-855 (2010).

\bibitem{das-81} Das, P. S., Fuzzy groups and level subgroups, J. Math.
Anal. Appl. 84 (1981) 264-269.

\bibitem{dix-90} Dixit, V. N., Kumar, R. and Ajamal, N., Level subgroups and
union of fuzzy subgroups, Fuzzy Sets Syst. 37 (1990)359-371.

\bibitem{dummit-04} Dummit D. S. and Foote, R. M., Abstract Algebra, John
Wiley\&Sons, Inc., 2004.

\bibitem{fen-08} Feng, F., Jun Y.B. and Zhao, X., Soft semirings, Comput.
Math. Appl. 56, 2621-2628,(2008).

\bibitem{isaacs-09} Isaacs, I. M., Algebra, American Mathematical Society,
2009.

\bibitem{jun-08-1} Jun, Y.B., Soft BCK/BCI-algebras, Comput. Math. Appl. 56,
1408-1413 (2008).

\bibitem{jun-08-2} Jun Y.B. and Park, C.H., Applications of soft sets in
ideal theory of BCK/BCI-algebras, Inform. Sci. 178, 2466-2475 (2008).

\bibitem{jun-10} Jun, Y.B., Lee K.J. and Khan, A., Soft ordered semigroups,
Math. Logic Quart. 56/1, 42-50 (2010).

\bibitem{kay-11} Kayg\i s\i z, K., On Soft int-Groups, Ann. Fuzzy Math.
Inform., 4(2), (2012) 365-375.

\bibitem{kim-94} Kim, J. G., Fuzzy orders relative to fuzzy subgroups,
Inform. Sci. 80 (1994) 341-348. 31.

\bibitem{kum-92} Kumar, I. J., Saxena P. K. and Yadav, P., Fuzzy normal
subgroups and fuzzy quotients, Fuzzy Sets Syst. 46 (1992) 121-132.

\bibitem{liu-82} Liu, W. J., Fuzzy invariant subgroups and fuzzy ideals,
Fuzzy Sets Syst. 8 (1982)133-139.

\bibitem{maj-03} Maji, P.K., Biswas R. and Roy, A.R., Soft set theory,
Comput. Math. Appl. 45, 555-562 (2003).

\bibitem{mol-99} Molodtsov, D.A., Soft set theory-first results, Comput.
Math. Appl. 37, 19-31 (1999).

\bibitem{mordeson-05} Mordeson, J. N., Bhutani K. R. and Rosenfeld, A.,
Fuzzy Group Theory, Springer, 2005.

\bibitem{mor-94} Morsi N. N. and Yehia, S. E., Fuzzy-quotient groups, Inf.
Sci. 81 (1994) 177-191.

\bibitem{muj-86} Mujherjee N. P. and Bhattacharya, P., Fuzzy Groups Some
Group-Theoretic Analogs, Information Science39,247-268 (1986).

\bibitem{muk-84} Mukherjee N. P. and Bhattacharya, P., Fuzzy normal
subgroups and fuzzy cosets, Inform. Sci. 34 (1984) 225-239.

\bibitem{ros-71} Rosenfeld, A., Fuzzy groups, J. Math. Anal. Appl. 35 (1971)
512-517.

\bibitem{sez-11} Sezgin A. and Atag\"{u}n, A.O., On operations of soft sets,
Comput. Math. Appl. 61/5, 1457-1467 (2011).

\bibitem{zad-65} Zadeh, L. A., Fuzzy sets, Inform. Control 8 (1965) 338-353.
\end{thebibliography}
\end{document}